\RequirePackage{fix-cm}
\documentclass[smallextended]{svjour3}       
\smartqed  
\usepackage{graphicx, amsfonts}
\usepackage{hyperref}

\usepackage{pgf,tikz} 
\usepackage{mathrsfs} 
\usetikzlibrary{arrows} 
\usepackage{adjustbox}

\newcommand{\Godel}{G\"odel}

\newcommand{\Lowenheim}{L\"owenheim}

\newcommand{\Poincare}{Poincar\'{e}}
\renewcommand{\P}{{\mathbb P}}

\newcommand{\Vbar}{{\overline{V}}}

\newcommand{\of}{\subseteq}

\newcommand{\set}[1]{\{\,{#1}\,\}}

\newcommand{\Mantle}{{\mathord{\rm M}}}

\def\<#1>{\langle#1\rangle}

\newcommand{\ZFC}{{\rm ZFC}}
\newcommand{\ZF}{{\rm ZF}}

\newcommand{\scrA}{{\mathcal A}}
\begin{document}

\title{From geometry to geology
%
}
\subtitle{An invitation to mathematical pluralism through the phenomenon of  independence}


\author{Jonas Reitz}


\institute{J. Reitz \at
              The New York City College of Technology, CUNY\\
              300 Jay Street, Brooklyn, NY 11201 \\
              Tel.: +1-718-260-5380\\
              \email{jreitz@citytech.cuny.edu}           
}

\date{Received: date / Accepted: date}

\maketitle

\begin{abstract}
This paper explores how a pluralist view can arise in a natural way out of the day-to-day practice of modern set theory.  By contrast, the widely accepted orthodox view is that there is an ultimate universe of sets $V$, and it is in this universe that mathematics takes place.  From this view, the purpose of set theory is ``learning the truth about $V$.''  It has become apparent, however, that the phenomenon of independence - those questions left unresolved by the axioms - holds a central place in the investigation.  This paper introduces the notion of independence, explores the primary tool (``soundness'') for establishing independence results, and shows how a plurality of models arises through the investigation of this phenomenon.  Building on a familiar example from Euclidean geometry, a template for independence proofs is established.  Applying this template in the domain of set theory leads to a consideration of forcing, the tool par excellence for constructing universes of sets.  Fifty years of forcing has resulted in a profusion of universes exhibiting a wide variety of characteristics - a multiverse of set theories.  Direct study of this multiverse presents technical challenges due to its second-order nature.  Nonetheless, there are certain nice ``local neighborhoods'' of the multiverse that are amenable to first-order analysis, and \emph{set-theoretic geology} studies just such a neighborhood, the collection of grounds of a given universe $V$ of set theory.  I will explore some of the properties of this collection, touching on major concepts, open questions, and recent developments.

\end{abstract}

\section{Introduction}
\label{Section.Introduction}

This paper is about how a pluralist view can arise in a natural way out of the day-to-day practice of modern set theory.  The tools and techniques common in set theory today act to acclimatize the practitioner to easy movement between set-theoretic universes.  As in any discipline, the day-to-day practice exerts a gentle but undeniable influence on the philosophical viewpoint, and, while this has not (yet) engendered a general embracement of the multiverse view of set theory, it has at the least influenced the ``working perspective.'' To paraphrase Davis and Hersch \cite{DavisHersh1981:TheMathematicalExperience}, the average working set theorist may hold the universe view on Sunday, but she is a multiversist during the week - or at least, she behaves so.

How has this situation arisen?  Why is the work of set theory so inextricably tied to the construction and investigation of diverse set-theoretic universes?  To address this question, I will explore an idea that has become absolutely central to set theory over the past century, the phenomenon of independence.  This notion will be introduced and defined for the layman, and a strategy for establishing independence results will be discussed in detail.  Some time will be spent on an extended example of an independence proof taken from a standard high school mathematical context, Euclidean geometry.  The lessons of this example will be generalized to lay the groundwork for independence results in set theory, motivating the study of alternative universes as a central theme. Finally, the technique of \emph{forcing} will be introduced as a primary universe-building tool, employed in nearly every set-theoretic independence result of the past half-century.  Careful examination of the mechanics of forcing lead to \emph{set-theoretic geology}, the study of a particular nice neighborhood of the set-theoretic multiverse - a first-order realization of the pluralistic view of set-theory.

Who is this paper for?  If you are inclined towards philosophy of mathematics and would like to see how a pluralist perspective might arise out of traditional mathematical practice, then this paper could be for you.  If you have some mathematical fluency but know little of logic or set theory, and would like a quick-and-dirty motivation for the work of modern set theorists, you may find just such an (incomplete and unsatisfying) explanation here.  If you are a working set theorist you have no doubt arrived at your own conclusions on these matters - but you may find some entertainment in exploring the story I have to tell.  In the event that you are misrepresented here, I offer my regrets.

Let's begin with a non-pluralist perspective.

\section{Independence}
\label{Section.Independence}

\subsection{The orthodox view of set theory}

According to the widely accepted orthodox view of set theory, there is an ultimate universe of sets, $V$, in which mathematics takes place.  All mathematical objects - natural numbers, real numbers, all the varieties of groups, topological spaces, etc. -  exist in $V$.  According to this view, the primary goal of set theory is ``learning the truth about $V$.''

The underlying axioms of $V$ are the $\ZFC$ axioms, the Zermelo-Fraenkel axioms together with the axiom of choice.  Set theory studies the consequences of the $\ZFC$ axioms, and according to the orthodox view, we are thus uncovering the truth about the mathematical universe $V$.

\subsection{The troubling phenomenon of independence}

The orthodox view of set theory suggests a certain attitude when presented with new set-theoretic ideas: while we may not know the truth of the matter at the outset, hard work and cleverness should in principle allow us to determine it.  Given a proposition about sets, either it is true or it is false, and the role of the set theorist is to discover which is the case.  The nature of the proposition can vary widely, and is limited only by what we can express in the language of set theory -- first order logic, with the membership relation $\in$ -- no great limitation, as this language is highly, even staggeringly, expressive.  If we refer to such a proposition by the letter $P$, then (for example) $P$ might be the statement ``every set of ordinals has a least element'' (which is true, a basic fact about ordinals), or ``there exists a set $x$ which is a member of itself'' (false, as it contradicts the axiom of foundation).  Faced with such a $P$, a natural (and, as it happens, distressingly ineffective) approach is to bring to bear the core process of mathematics, that of \emph{proof}: we begin with the axioms, and from them we try to prove $P$ (or, in the case that we suspect $P$ is false, we try to disprove $P$ -- put another way, we try to prove the negation $\neg P$).  This process comprises a great deal of the day-to-day activity of mathematicians, and it is an idiosyncratic and highly personal process, filled with painstaking detail work, wild leaps of creativity, meticulous notations, hastily scrawled diagrams, long walks, and staring again and again at past notes and dog-eared references.  At the end of the day, sometimes we have proved $P$, and sometimes we have proved $\neg P$, and sometimes we are simply forced to set down our pencils and try again tomorrow.  (Aside: Rohit Parikh, one of my graduate professors and a great logician, once commented that a reasonable strategy, in the absence of compelling reasons to either believe or disbelieve $P$, is to spend odd-numbered days working on a proof of $P$ and even-numbered days working on a proof of $\neg P$.)  Most mathematicians know all too well the experience of repeated failure, setting down the pencil, day after day, and returning to gnaw on the problem again.  It is at this point in the tale that the phenomenon of \emph{independence} raises its head, terrible and marvelous, and a change in perspective begins to present itself.

\begin{definition}[independence]
A proposition $P$ is \textbf{independent} of the axioms $\scrA$ if $P$ can be neither proved nor disproved from $\scrA$.
\end{definition}

A note on notation - I will use capital letters $P$ and $Q$ to represent mathematical propositions (or statements, the terms can be taken interchangeably), and while I will generally give examples in plain English it is assumed that any such proposition could be written out in strict, symbolic, first-order logic under sufficiently dire circumstances  (in much the way that pseudocode can be translated into actual code, and thence into machine language).  I will use the script letter $\scrA$ to represent a set of axioms - where an axiom is just another mathematical proposition, and $\scrA$ is a collection of such propositions that we have distinguished by taking them to be true at the outset.  If we are discussing  geometry then $\scrA$ might consist of Euclid's postulates, in the case of group theory  $\mathcal A$ would be the three group axioms, in set theory $\mathcal A$ might be the \ZFC\ axioms, and so on.

Also note that the notion of independence is not a property of $P$ alone, but defines a relationship between $P$ and $\scrA$ - and, while we may often make utterances such as ``$P$ is independent,'' there is always an underlying axiom system $\scrA$ implied.  

When a proposition $P$ is independent of your axioms, then \emph{all} attempts to prove or disprove $P$ from $\scrA$ are doomed to fail.  Of course, from the perspective of the working mathematician, it is not at all clear whether a given $P$ is independent of the axioms.  Repeated failures to prove or disprove $P$ may add emotional weight to the argument, but little more - it might simply be the case that proving $P$ requires some new approach or insight, and it is for this reason that mathematicians carry on in the face of repeated failure.  On the other hand, \Godel's Incompleteness Theorems showed us that the phenomenon of independence is absolutely ubiquitous, arising naturally for any sufficiently powerful collection of axioms.  (Aside: This suggests an update to Parikh's strategy - In the absence of compelling evidence regarding the truth of $P$, spend odd-numbered days working on a proof of $P$, even-numbered days working on a proof of $\neg P$, and weekends working on a proof that $P$ is independent.)  

How are we to proceed?  Faced with a seemingly intractable proposition $P$, and a growing suspicion that $P$ may indeed be independent of the axioms, what strategy can we employ to establish independence rigorously?  How do we go about \emph{proving} that $P$ is independent of $\scrA$?

\subsection{The soundness property: proving independence by building models}
\label{Subsection.ProvingIndependence}

At its heart, an independence proof must establish two negative statements:

\begin{enumerate}
    \item {That there is no proof of $P$ from $\scrA$.}
    \item {That there is no proof of $\neg P$ from $\scrA$.}
\end{enumerate}

A (non-rigorous, but intuitively appealing) rule of thumb in mathematics: Negative statements are hard!  To show something exists, one can simply work towards finding an example.  To show something does \emph{not} exist, the strategy is less clear.  In the case of independence, how can we establish that ``no proof exists''?  A typical approach to negative statements, familiar to any undergraduate student of proofs and logic, is ``assume it does exist and show a contradiction.''  This is an attractive and useful notion, but it is not clear how it applies in this case:  Let us assume we have a proof of $P$.  What contradiction might arise?  There does not seem to be much to get our hands on.  However, we have a tremendous tool at our disposal, that of \emph{soundness}.

\begin{definition}[soundness]
If there is a proof of $P$ from $\scrA$, then every structure satisfying $\scrA$ must satisfy $P$.
\end{definition}

Soundness represents the most fundamental connection between \emph{syntax}, the strings of written symbols on the page that comprise our propositions and proofs, and \emph{semantics}, the meaning we ascribe to those symbols when we consider the groups, or sets, or geometric objects they describe.  The shorthand slogan for soundness might be ``if you can prove it, it must be true.''  This is absolutely essential to the business of mathematics - we search for proofs exactly because we'd like to establish the truth of a given proposition.  In group theory, as an example, if we've proved that a property $P$ follows from the group axioms $\scrA$, then we know that $P$ must be true in all groups.  This is soundness.

Soundness connects proofs and structures - this is nice, because the human imagination has a great capacity for imagining new structures, and we can employ this capacity to our advantage in establishing facts about proofs.  In particular, an equivalent formulation of soundness (take the contrapositive!) yields the following principle, which will form the foundation of our independence proofs:

\begin{theorem}If we can exhibit a structure satisfying $\scrA$ in which $\neg P$ holds, then there is no proof of $P$ from $\scrA$.
\end{theorem}

This gives us a concrete ``hands-on'' strategy for establishing the two negative statements of an independence proof:

\begin{corollary}\label{Corollary.ProvingIndependence}
If we can exhibit two structures, both satisfying $\scrA$, such that $P$ holds in one structure and $\neg P$ holds in the other, then $P$ is independent of $\scrA$.
\end{corollary}

\subsection{An extended example: the independence of the parallel postulate}
\label{Subsection.IndependenceExample-EuclidsFifthPostulate}

Ultimately I am a set theorist, and a great deal of the current work in set theory comes down to independence proofs - showing that some statement $P$ about sets is independent of the \ZFC\ axioms.  In order to make more concrete the day-to-day business of independence proofs, I'd like to spend some time exploring an example in a more familiar context, an area of mathematics that you may recall from your own secondary school mathematics curriculum - that of Euclidean geometry.  This example will not constitute a rigorous proof of independence, but may at least serve as an informal proof - an solid attempt to convince the reader, with enough content to point the determined student in the right direction if he or she wishes to pursue the matter.  It will also (more importantly) illustrate the broad structure and many key elements of an independence proof, and will serve as a model from which we can discuss independence in other contexts.

\subsubsection{Euclidean geometry}\label{Subsubsection.Euclideangeometry}

In the United States, Euclidean geometry is generally presented at the high school level and constitutes the student's first (and in many cases, only) introduction to the form and practice of formal proof. The topic is introduced via Euclid's postulates, the axioms that form the basis for plane geometry.

\begin{definition}[The five postulates of Euclidean geometry]

Primitive (undefined) notions: point, line.

\begin{enumerate}
    \item {A line segment can be drawn joining any two points.}
    \item {Any line segment can be extended indefinitely in a line.}
    \item {Given any line segment, a circle can be drawn having the segment as radius and one endpoint as center.}
    \item {All right angles are congruent.}
    \item {(Parallel Postulate)  Given a line and a point not on the line, there is exactly one line passing through the point parallel to the given line.}

\end{enumerate}
\end{definition}


The discussion below will focus on the fifth postulate on the list, the parallel postulate, which has a long and contentious history.  It should be noted that this presentation is not strictly correct - the fifth postulate stated by Euclid was not the parallel postulate, but a slightly-more-complicated, equivalent-under-certain-reasonable-assumptions statement - a distinction that we will ignore in this example.  The controversy of the parallel postulate is whether it is, strictly speaking, necessary.  Perhaps it is simply a consequence of the first four postulates, in which case it can safely be eliminated from our list of axioms (which, in the interest of simplicity, aesthetics, and so on, we'd like to keep as compact as possible).   Many, many attempts were made to find a proof of the parallel postulate from the first four postulates.  In desperation, some attempts were even made to prove the parallel postulate false.  In the first half of the 19th century it was finally established that these attempts were doomed to fail from the start, when the parallel postulate was shown to be independent of the other postulates (the history is complicated, but the first published results exhibiting models of geometry in which the parallel postulate failed were produced independently by Bolyai and Lobachevsky around 1830).  Let us look in more detail at a proof of this result, gleaning what we can from the methods employed.

Before we begin, there is one issue that needs to be addressed with a little care - the meaning of the word ``parallel.''  This is a word that carries strong intuitive associations (two parallel lines never meet, are always the same distance apart, have the same perpendiculars, have equal corresponding angles when cut by a transverse, and so on).  In the most familiar context, the Euclidean plane, these various notions all coincide - but as we move to other models, we will find that they need not always be the same.  For example, we may find two lines that never cross, but a perpendicular to one may not be perpendicular to the other - is such a pair of lines parallel? It depends on the definition.  Because of this, we need to carefully define which of these various notions we will refer to as ``parallel'' - the accepted definition is:

\begin{definition}
By definition, two lines are \emph{parallel} if they do not intersect.
\end{definition}

\subsubsection{The independence of the parallel postulate}\label{Subsubsection.Independenceofparallelpostulate}

We seek to establish the following theorem:

\begin{theorem}
The parallel postulate is independent of postulates 1-4.
\end{theorem}

\begin{proof}

We will follow the template outlined in Corollary \ref{Corollary.ProvingIndependence}.  Let us take the axioms $\scrA$ to consist of Euclid's postulates $1-4$, and the proposition $P$ to be the parallel postulate.  To complete the proof, we must:
\begin{itemize}
\item {
    exhibit a model of $\scrA$ in which $P$ holds, and
    }
\item {
    exhibit a model of $\scrA$ in which $P$ fails.
    }
\end{itemize}

The first model will be the familiar Euclidean plane (or just ``the plane'').  I will not describe it from first principles but will instead rely on the model in your head, asking you to reason about this familiar object via your intuitive understanding.  Because of this, the resulting ``proofs'' will not be reduced to basic principles, but will once again be reduced only to your intuitive notions about the plane - we will consider a proposition ``proved'' if it seems intuitively obvious, or can be reduced to principles that seem intuitively obvious.  The second model will be somewhat more technical, and I will describe it in more technical detail - but ultimately, the associated proof will also rest on your geometric intuitions (though you will have to do a little more work to apply them).

\subsubsection{The Euclidean plane}\label{Subsubsection.Euclideanplane}

\begin{lemma}[A model of $P$]\label{Lemma.Euclideanplane}
The standard Euclidean plane, with points and lines interpreted in the familiar way, satisfies Euclid's postulates 1-4 as well as the parallel postulate.
\end{lemma}

\begin{proof}

Our intuitive understanding of the Euclidean plane (or simply ``the plane,'' that flat, boundless, two-dimensional surface that we freely represent on paper, chalkboards, bar napkins, computer screens, and so on) is so strong that the proof of this lemma seems entirely obvious - so much so, that it is unsatisfying.  For example, if you show me two points, it is ``obvious'' that I can draw a line segment joining them (postulate 1) - I can picture it clearly,  there on the paper!  There seems to be no foundation upon which to rest these ``obvious'' facts (akin to asking ``why?'' when presented with the basic arithmetical fact $2+2=4$).  The problem here is that we have not given a rigorous definition of the Euclidean plane - we have not provided any formal basis for our arguments.  Let this pass for now, noting only that while it \emph{can} be done, the presentation adds a great deal of formal complexity while contributing little to our intuitive understanding of the picture.

Proof of Lemma \ref{Lemma.Euclideanplane}:

\begin{itemize}
\item Given any two points, we can draw a line segment between them, and so the Euclidean plane satisfies postulate 1 (Fig. \ref{fig.Euclidean1}).
\item Given a line segment, we can extend it in both directions to obtain a line (postulate 2, Fig. \ref{fig.Euclidean2}). 

\begin{figure}[h]
  \centering
  \begin{minipage}[t]{0.3\textwidth}
       \begin{adjustbox}{width=\textwidth} 
        \begin{tikzpicture}[line cap=round,line join=round,>=triangle 45,x=1.0cm,y=1.0cm]
            \clip(0,0) rectangle (4,4);
            \draw [dash pattern=on 5pt off 5pt] (1,3)-- (3,1);
            \begin{scriptsize}
            \draw [fill=black] (1,3) circle (2.5pt);
            \draw (1.2,3.3) node {$A$};
            \draw [fill=black] (3,1) circle (2.5pt);
            \draw (3.3,1.2) node {$B$};
            \end{scriptsize}
        \end{tikzpicture}
        \end{adjustbox}
    \caption{The dashed line segment connects two given points.}
    \label{fig.Euclidean1}
  \end{minipage}
  \hspace{2cm}
  \begin{minipage}[t]{0.3\textwidth}
       \begin{adjustbox}{width=\textwidth} 
        \begin{tikzpicture}[line cap=round,line join=round,>=triangle 45,x=1.0cm,y=1.0cm]
            \clip(0,0) rectangle (4,4);
            \draw (1,3)-- (3,1);
            \draw [dash pattern=on 5pt off 5pt] (0,4)-- (1,3);
            \draw [dash pattern=on 5pt off 5pt] (3,1)-- (4,0);
            \begin{scriptsize}
            \draw [fill=black] (1,3) circle (2.5pt);
            \draw (1.2,3.3) node {$A$};
            \draw [fill=black] (3,1) circle (2.5pt);
            \draw (3.3,1.2) node {$B$};
            \end{scriptsize}
        \end{tikzpicture}
        \end{adjustbox}
    \caption{The dashed line extends the given line segment.}
    \label{fig.Euclidean2}
  \end{minipage}
\end{figure}

\item Given a line segment, we can use a standard protractor to draw a circle having the segment as a radius and one endpoint as center (postulate 3, Fig. \ref{fig.Euclidean3}).

\begin{figure}[ht]
  \centering
  \begin{minipage}[t]{0.3\textwidth}
       \begin{adjustbox}{width=\textwidth} 
        \begin{tikzpicture}[line cap=round,line join=round,>=triangle 45,x=1.0cm,y=1.0cm]
        \clip(0,0) rectangle (4,4);
        \draw [line width=1.6pt] (1,3)-- (2,2);
        \draw [dash pattern=on 5pt off 5pt] (2,2) circle (1.41421356237cm);
        \begin{scriptsize}
        \draw [fill=black] (1,3) circle (2.5pt);
        \draw [fill=black] (2,2) circle (2.5pt);
        \end{scriptsize}
        \end{tikzpicture}
        \end{adjustbox}
    \caption{The dashed circle has the given segment as a radius.}
    \label{fig.Euclidean3}
  \end{minipage}
  \hfill
  \begin{minipage}[t]{0.3\textwidth}
       \begin{adjustbox}{width=\textwidth} 
        \begin{tikzpicture}[line cap=round,line join=round,>=triangle 45,x=1.0cm,y=1.0cm]
        \clip(0,0) rectangle (4,4);
        \draw[color=black,fill=black,fill opacity=0.1] (1.8,1.9) -- (2,2) -- (1.9,2.2) -- (1.7,2.1)  -- cycle; 
        \draw (1,1.5)-- (2,2);
        \draw (2,2)-- (1.5,3);
        \draw[color=black,fill=black,fill opacity=0.1] (2.65, 1.63) -- (2.6,1.4) -- (2.83, 1.35) -- (2.87,1.58)  -- cycle; 
        \draw (2.8,2.4)-- (2.6,1.4);
        \draw (2.6,1.4)-- (3.6,1.2);
        \end{tikzpicture}
        \end{adjustbox}
    \caption{Right angles are congruent.}
    \label{fig.Euclidean4}
  \end{minipage}
  \hfill
  \begin{minipage}[t]{0.3\textwidth}
       \begin{adjustbox}{width=\textwidth} 
        \begin{tikzpicture}[line cap=round,line join=round,>=triangle 45,x=1.0cm,y=1.0cm]
        \clip(0,0) rectangle (4,4);
        \draw (0,3)-- (2,0);
        \draw [dash pattern=on 5pt off 5pt] (1,3)-- (3,0);
        \begin{scriptsize}
        \draw [fill=black] (2,1.5) circle (2.5pt);
        \draw (2.2,1.8) node {$A$};
        \end{scriptsize}
        \end{tikzpicture}
        \end{adjustbox}
    \caption{The dashed line is parallel to the given line and passes through the given point.}
    \label{fig.Euclidean5}
  \end{minipage}  
\end{figure}

\item Given two right angles, we can slide one over the top of the other, and we will see that they align perfectly - that is, all right angles are congruent (postulate 4, Fig. \ref{fig.Euclidean4}).

\item Given a line and a point not on the line, it is clear that we can draw a line parallel to the first, through the point - just picture it there! Or, to be more methodical, let us imagine any line at all passing through the point, and then adjust the direction of that line (rotate it) until it is parallel to the first.  You will notice there is exactly one ``sweet spot'' in the rotation in which the lines become parallel - it is this fact that demonstrates that only one such parallel line exists through the given point (parallel postulate, Fig. \ref{fig.Euclidean5}).
\end{itemize}

If you find these arguments compelling, then we have completed the first half of our independence proof - exhibiting a model in which $\scrA$ holds, and $P$ is true.  If you are not yet convinced, you might consider playing around with geometric constructions over at GeoGebra (\url{https://www.geogebra.org/geometry}), or, for the rigorous approach, take a look at Hilbert's classic work Foundations of Geometry \cite{Hilbert:FoundationsOfGeometry}, which recasts Euclidean Geometry on a firm formal foundation.

\qed
\end{proof}

\subsubsection{The \Poincare\ disk}\label{Subsubsection.Poincaredisk}

For the second half of our theorem, we will need to venture beyond the comfort of our familiar Euclidean plane into a more exotic model of the axioms of Euclid's postulates 1-4.  How do we get our hands on such a thing?  We use a classic strategy of mathematicians - start with a familiar structure, and see how we can modify it to produce something new...  In this case, we'll start with the Euclidean plane (or, properly, a subset of it - the unit disk), and will give a new interpretation of the primitive notions of point and line.  This model, the Poincar\'{e} disk, was introduced by Eugenio Beltrami in 1868 \cite{Beltrami1868}.

\begin{definition}[The \Poincare\ disk]

\begin{itemize}
\item {
    The universe consists of the unit disk (all those points on the plane whose distance from the origin is less than $1$).
    }
\item {
    ``Points'' are points in the unit disk.
    }
\item {
    ``Lines'' are segments of circles that intersect the boundary of the universe (the unit circle) perpendicularly.
    }
\end{itemize}
\end{definition}

Notice that the most fundamentally odd aspect of the \Poincare\ disk is the interpretation of ``lines.''  In coming to terms with a nonstandard definition of this type, it helps to recognize at the start that we are doing something in direct opposition to our intuition - as we explore this model further we will have much discussion of lines, and we must be on guard to suppress our carefully and deeply cultivated intuitions about this word. The importance of lines is not that they are indefinitely long in each direction, and ``straight,'' and so on - the importance is the relationships that hold among our primitive notions, points and lines. 

In getting a feeling for this new idea of lines, it's worth noticing that not every segment of a circle represents a line - it must intersect the boundary perpendicularly.  On the other hand, the inherent symmetry of circles shows that any segment perpendicular to the unit circle on one end, will also be perpendicular on the other end.  A final note about lines - an actual straight line segment passing through the center of the disk will also constitute a line in the sense of our \Poincare\ disk (note it is perpendicular to the boundary on either side) - we can reconcile this with our definition by thinking of it as a ``circle of infinite radius,'' or we can simply add it to the definition of line as an additional possibility.

One of the best ways to build your intuition about the \Poincare\ disk is to simply play around with it - and, while this is possible (but somewhat laborious) with traditional straight-edge-and-compass, modern technology provides a number of friendly alternatives - I recommend NonEuclid (\url{https://www.cs.unm.edu/~joel/NonEuclid/NonEuclid.html}) for an interactive tour of the \Poincare\ disk, and Hyperbolic Geometry in the \Poincare\ Disc (\url{https://www.geogebra.org/m/R5e9AggU}) for a more full-featured playground environment.

\begin{lemma}[A model of $\neg P$]
The \Poincare\ disk satisfies Euclid's postulates 1-4, but does not satisfy the parallel postulate.
\end{lemma}

\begin{proof}
The proof will utilize a combination of pictures, words, our newly-developing intuition about the \Poincare\ disk, hand-waving, redirection and blatant appeals to authority.

 \begin{figure}[ht]
  \centering
  \begin{minipage}[t]{0.3\textwidth}
        \begin{adjustbox}{width=\textwidth} 
        \begin{tikzpicture}[line cap=round,line join=round,>=triangle 45,x=2.0cm,y=2.0cm]
        \clip(-1.023968781312896,-1.0228225671654467) rectangle (1.041708718858042,1.0428549330054961);
        \draw(0.,0.) circle (2.cm);
        \draw [shift={(-1.1671490911322815,1.0230626589849616)},dash pattern=on 4pt off 4pt]  plot[domain=5.171649604390854:5.8239246831734235,variable=\t]({1.*1.1869684937437437*cos(\t r)+0.*1.1869684937437437*sin(\t r)},{0.*1.1869684937437437*cos(\t r)+1.*1.1869684937437437*sin(\t r)});
        \begin{scriptsize}
        \draw [fill=black] (-0.6409832092890928,-0.040912749529949305) circle (2.5pt);
        \draw [fill=black] (-0.1031736826173693,0.4968967771417755) circle (2.5pt);
        \end{scriptsize}
        \end{tikzpicture}
        \end{adjustbox}
    \caption{The dashed line segment connects two given points.}
    \label{fig.Poincare1}
  \end{minipage}
  \hfill
  \begin{minipage}[t]{0.3\textwidth}
        \begin{adjustbox}{width=\textwidth} 
        \begin{tikzpicture}[line cap=round,line join=round,>=triangle 45,x=2.0cm,y=2.0cm]
        \clip(-1.023968781312896,-1.0228225671654467) rectangle (1.041708718858042,1.0428549330054961);
        \draw(0.,0.) circle (2.cm);
        \draw [shift={(-1.1671490911322815,1.0230626589849616)}] plot[domain=5.171649604390854:5.8239246831734235,variable=\t]({1.*1.1869684937437437*cos(\t r)+0.*1.1869684937437437*sin(\t r)},{0.*1.1869684937437437*cos(\t r)+1.*1.1869684937437437*sin(\t r)});
        \draw [shift={(-1.1671490911322848,1.0230626589849632)},dash pattern=on 4pt off 4pt]  plot[domain=5.8239246831734235:6.263592470054243,variable=\t]({1.*1.1869684937437475*cos(\t r)+0.*1.1869684937437475*sin(\t r)},{0.*1.1869684937437475*cos(\t r)+1.*1.1869684937437475*sin(\t r)});
        \draw [shift={(-1.1671490911322815,1.023062658984961)},dash pattern=on 4pt off 4pt]  plot[domain=4.863365559655842:5.171649604390855,variable=\t]({1.*1.1869684937437432*cos(\t r)+0.*1.1869684937437432*sin(\t r)},{0.*1.1869684937437432*cos(\t r)+1.*1.1869684937437432*sin(\t r)});
        \begin{scriptsize}
        \draw [fill=black] (-0.6409832092890928,-0.040912749529949305) circle (2.5pt);
        \draw [fill=black] (-0.1031736826173693,0.4968967771417755) circle (2.5pt);
        \end{scriptsize}
        \end{tikzpicture}
        \end{adjustbox}
    \caption{The dashed line extends the given line segment.  Note the perpendicular intersections of the dashed line with the boundary.}
    \label{fig.Poincare2}
  \end{minipage}
   \hfill
  \begin{minipage}[t]{0.3\textwidth}
       \begin{adjustbox}{width=\textwidth} 
       \begin{tikzpicture}[line cap=round,line join=round,>=triangle 45,x=2.0cm,y=2.0cm]
        \clip(-1.023968781312896,-1.0228225671654467) rectangle (1.0417087188580425,1.0428549330054961);
        \draw(0.,0.) circle (2.cm);
        \draw [shift={(1.1857279038427406,-0.2444638054990037)}] plot[domain=2.4254359826530285:3.1836383610848906,variable=\t]({1.*0.6824318384647315*cos(\t r)+0.*0.6824318384647315*sin(\t r)},{0.*0.6824318384647315*cos(\t r)+1.*0.6824318384647315*sin(\t r)});
        \draw [dash pattern=on 4pt off 4pt] (0.335518150312748,-0.18187435463078386) circle (1.0218822832541252cm);
        \begin{scriptsize}
        \draw [fill=black] (0.503899192186319,-0.27314868150183047) circle (2.5pt);
        \draw [fill=black] (0.6709460906222331,0.2035461262299256) circle (2.5pt);
        \end{scriptsize}
        \end{tikzpicture}
        \end{adjustbox}
    \caption{The dashed circle has the given segment as a radius.  Note that the center of the circle in the \Poincare\ disk is different from the center of the circle as drawn on the plane.}
    \label{fig.Poincare3}
  \end{minipage}
\end{figure}

\begin{itemize}
    \item Given two points, we can imagine any line passing through the first point - now rotate that line until it also passes through the second point.  Thus the \Poincare\ disk satisfies postulate 1 (Fig. \ref{fig.Poincare1}.
    
    \item Any line segment is, of necessity, simply a part of a complete line - just extend the ends until the boundary is reached (postulate 2, Fig. \ref{fig.Poincare2}).
   
    \item Given a line segment, we want to produce a circle with that segment as its radius.  This one is a bit of a rabbit-hole, whose investigation provokes many more questions:
    \begin{itemize}
        \item What does a circle look like in the \Poincare\ disk? 
        \item What do we mean by distance?  
        \item What can we say about angles? and so on...  
    \end{itemize}
    Here, the argument will depend on an extremely important (and not obvious) property of the \Poincare\ disk, that it \emph{faithfully represents angles} (it is a ``conformal model'').  If two lines in the \Poincare\ disk intersect in what appears to us to be right angles, then they do, in fact, intersect in right angles (under the correct definition of angle interpreted in the \Poincare\ disk). This extends also to circles, so if we draw a circle inside the \Poincare\ disk, it will also satisfy the definition of circle interpreted within the model (``the set of all points a given distance from the center'') -- circles look like circles, although the center of the circle (from our perspective) may not be the same as the center of the circle (defined inside the model).  Nonetheless, given any line segment inside the \Poincare\ disk, we can indeed construct a circle with that segment as one of its radii (postulate 3, Fig. \ref{fig.Poincare3}).  
 
    \item As a conformal model, the \Poincare\ disk faithfully represents angles - and so all right angles are congruent (postulate 4, Fig. \ref{fig.Poincare4}).
    
    \item Given a line and a point not on the line, we can easily construct many different lines, passing through the point, and parallel to the first line (recall our definition of parallel - they ``do not meet'').  Thus in the \Poincare\ disk the parallel postulate is false (Fig. \ref{fig.Poincare5}).
    
\end{itemize}
\qed
\end{proof}

\begin{figure}[ht]
  \centering
  \begin{minipage}[t]{0.3\textwidth}
       \begin{adjustbox}{width=\textwidth} 
        \begin{tikzpicture}[line cap=round,line join=round,>=triangle 45,x=2.0cm,y=2.0cm]
        \clip(-1.023968781312896,-1.0228225671654467) rectangle (1.0417087188580425,1.0428549330054961);
        \draw(0.,0.) circle (2.cm);
        \draw [shift={(1.1857279038427406,-0.2444638054990037)}] plot[domain=2.4254359826530285:3.1836383610848906,variable=\t]({1.*0.6824318384647315*cos(\t r)+0.*0.6824318384647315*sin(\t r)},{0.*0.6824318384647315*cos(\t r)+1.*0.6824318384647315*sin(\t r)});
        \draw [dash pattern=on 4pt off 4pt] (0.335518150312748,-0.18187435463078386) circle (1.0218822832541252cm);
        \draw [shift={(1.681350082586509,0.6698500184794387)}] plot[domain=3.37924910127298:4.032400442754742,variable=\t]({1.*1.5085215104434422*cos(\t r)+0.*1.5085215104434422*sin(\t r)},{0.*1.5085215104434422*cos(\t r)+1.*1.5085215104434422*sin(\t r)});
        \draw [shift={(0.37640460676569837,-1.737487141260395)}] plot[domain=1.2625313611818774:1.9316155602031024,variable=\t]({1.*1.4698782242211978*cos(\t r)+0.*1.4698782242211978*sin(\t r)},{0.*1.4698782242211978*cos(\t r)+1.*1.4698782242211978*sin(\t r)});
        \draw [shift={(1.1857279038427355,-0.24446380549900493)}] plot[domain=3.1836383610848893:3.710781470144596,variable=\t]({1.*0.6824318384647264*cos(\t r)+0.*0.6824318384647264*sin(\t r)},{0.*0.6824318384647264*cos(\t r)+1.*0.6824318384647264*sin(\t r)});
        \begin{scriptsize}
        \draw [fill=black] (0.503899192186319,-0.27314868150183047) circle (2.5pt);
        \draw [fill=black] (0.6709460906222331,0.2035461262299256) circle (2.5pt);
        \end{scriptsize}
        \end{tikzpicture}
        \end{adjustbox}
    \caption{Same as Fig. \ref{fig.Poincare3} with additional radii shown.}
    \label{fig.Poincare3a}
  \end{minipage}
  \hfill
  \begin{minipage}[t]{0.3\textwidth}
       \begin{adjustbox}{width=\textwidth} 
        \begin{tikzpicture}[line cap=round,line join=round,>=triangle 45,x=2.0cm,y=2.0cm]
        \clip(-1.0299745523116648,-1.0166787720330333) rectangle (1.027642721156404,1.0316490329318073);
        \draw[fill=black,fill opacity=0.1] (-0.4769359849096636,-0.17507786692982913) -- (-0.41281371094678176,-0.1608284727158551) -- (-0.42706310516075574,-0.09670619875297322) -- (-0.49118537912363763,-0.11095559296694724) -- cycle; 
        \draw[fill=black,fill opacity=0.1] (0.058335465011074165,0.4846009222934291) -- (0.08052860521628462,0.546424670007944) -- (0.01870485750176973,0.5686178102131545) -- (-0.003488282703440729,0.5067940624986396) -- cycle; 
        \draw(0.,0.) circle (2.cm);
        \draw [shift={(-1.1996359678979487,-0.3383647586527364)}] plot[domain=-0.1443228843309834:0.3106053481078739,variable=\t]({1.*0.7440545446219482*cos(\t r)+0.*0.7440545446219482*sin(\t r)},{0.*0.7440545446219482*cos(\t r)+1.*0.7440545446219482*sin(\t r)});
        \draw [shift={(-3.055798566768707,7.87861561576776)}] plot[domain=5.022994328492556:5.075006591159147,variable=\t]({1.*8.391095810541408*cos(\t r)+0.*8.391095810541408*sin(\t r)},{0.*8.391095810541408*cos(\t r)+1.*8.391095810541408*sin(\t r)});
        \draw [shift={(0.3330344374765665,1.2422953617765313)}] plot[domain=4.283280871417195:4.807556751302325,variable=\t]({1.*0.8088323079828218*cos(\t r)+0.*0.8088323079828218*sin(\t r)},{0.*0.8088323079828218*cos(\t r)+1.*0.8088323079828218*sin(\t r)});
        \draw [shift={(-1.5821823103136188,1.229112852451698)}] plot[domain=5.854077198212091:6.0330325540846,variable=\t]({1.*1.7360931043959875*cos(\t r)+0.*1.7360931043959875*sin(\t r)},{0.*1.7360931043959875*cos(\t r)+1.*1.7360931043959875*sin(\t r)});
        \end{tikzpicture}
        \end{adjustbox}
    \caption{Right angles are congruent.}
    \label{fig.Poincare4}
  \end{minipage}
  \hfill
  \begin{minipage}[t]{0.3\textwidth}
       \begin{adjustbox}{width=\textwidth} 
        \begin{tikzpicture}[line cap=round,line join=round,>=triangle 45,x=2.0cm,y=2.0cm]
        \clip(-1.023968781312896,-1.0228225671654467) rectangle (1.041708718858042,1.0428549330054961);
        \draw(0.,0.) circle (2.cm);
        \draw [shift={(-1.930841600410906,-1.8998901469359994)}] plot[domain=0.39920948262639655:1.155427644302598,variable=\t]({1.*2.5174852246441963*cos(\t r)+0.*2.5174852246441963*sin(\t r)},{0.*2.5174852246441963*cos(\t r)+1.*2.5174852246441963*sin(\t r)});
        \draw [shift={(0.42019776409460263,1.520711250499339)},dash pattern=on 4pt off 4pt]  plot[domain=3.7562975883669623:5.129300708328169,variable=\t]({1.*1.2202986799736226*cos(\t r)+0.*1.2202986799736226*sin(\t r)},{0.*1.2202986799736226*cos(\t r)+1.*1.2202986799736226*sin(\t r)});
        \draw [shift={(1.276096647740023,0.46944832379988627)},dash pattern=on 4pt off 4pt]  plot[domain=2.6677711238703825:4.320439097192541,variable=\t]({1.*0.9213058032444205*cos(\t r)+0.*0.9213058032444205*sin(\t r)},{0.*0.9213058032444205*cos(\t r)+1.*0.9213058032444205*sin(\t r)});
        \draw [shift={(0.7277975402965705,1.142899984284227)},dash pattern=on 4pt off 4pt]  plot[domain=3.3152029677827315:4.975496960567746,variable=\t]({1.*0.9142809380811917*cos(\t r)+0.*0.9142809380811917*sin(\t r)},{0.*0.9142809380811917*cos(\t r)+1.*0.9142809380811917*sin(\t r)});
        \begin{scriptsize}
        \draw [fill=black] (0.370240086501367,0.30143560537439174) circle (2.5pt);
        \end{scriptsize}
        \end{tikzpicture}
        \end{adjustbox}
    \caption{Each dashed line is parallel to the given line and passes through the given point.}
    \label{fig.Poincare5}
  \end{minipage}
\end{figure}

Taking stock, we have now exhibited  models in which Euclid's postulates 1-4 hold and 
\begin{itemize}
\item {
    the parallel postulate holds (the Euclidean plane).
    }
\item {
    the parallel postulate fails (the \Poincare\ disk).
    }
\end{itemize}

Thus the parallel postulate is independent of postulates 1-4. 
\qed
\end{proof}

\subsubsection{Multiplicity in geometry}
\label{Subsubsection.MultiplicityInGeometry}

Where does this leave us?  We have considered two very different models, both satisfying the basic principles of geometry, but disagreeing about fundamental geometric properties. This is fascinating and provocative!  

A useful test question in this situation might be:

\begin{question}
    Is there a compelling reason to accept one model of geometry over the other?
\end{question}

Suppose the independence of the parallel postulate had been discovered in Euclid's time: one can imagine a prevalent argument that we should simply adopt the parallel postulate as a new axiom - after all, it holds in the familiar Euclidean plane, the ideal model of the universe.  This perspective (analagous to the orthodox view of set theory) holds implicit the notion that there is a ``right'' geometry, and our goal is to ferret out its properties.  The history of developments in geometry and physics, however, give this the air of a cautionary tale, suggesting that too rigid an attachment to this ideal may be dangerously limiting.  The embracement of multiplicity in geometry has been a fruitful one both in pure mathematics and in related sciences, and the investigation of alternative geometries has both informed and been informed by our our deepening understanding of the geometry of space-time. Our best understanding of the true geometry of the space is that it is neither that of the Euclidean plane nor of the \Poincare\ disk, but something entirely more complex.

From this perspective, a different natural question arises:

\begin{question}
    Can we find still more models of geometry, with diverse and interesting properties?
\end{question}

\section{Independence in Set Theory}
\label{Section.IndependenceinSetTheory}

Let us turn out attention back to set theory.  Here we have an underlying set of axioms, the \ZFC\ axioms (analogous to Euclid's postulates 1-4), and a general (though by no means universal) agreement among mathematicians that it these are the ``right'' axioms, in the sense that they correctly capture our basic intuitions about sets.  Taking these axioms as our background theory, what further facts can we prove about sets?

An instructive example comes from the earliest days of set theory, when Cantor, introducing the notion of cardinality, demonstrated that the size of the real numbers is strictly larger than the size of the natural numbers.  He then asked the obvious next question, ``Are there any sizes in between?''  This question (in the case of a negative answer, it is the continuum hypothesis), is one of the most famous in set theory.  Much as in the case of the parallel postulate, it was widely believed that the continuum hypothesis could simply be proven from \ZFC, and Cantor and many others devoted enormous time and effort to developing such a proof.  It was not until much later  that the combined efforts of \Godel\ and Cohen established once and for all: 

\begin{theorem}[\Godel, Cohen]
The continuum hypothesis is independent of \ZFC.
\end{theorem}

In the past century of the study of the \ZFC\ axioms, the question of independence has attained a place of central importance.

\begin{example}[Examples of independence in set theory]
\begin{itemize}
    \item {The continuum hypothesis is independent of \ZFC.}
    \item {The axiom of choice is independent of the Zermelo-Frankel axioms.}
    \item {The existence of inaccessible cardinals, measurable cardinals, and all of the other large cardinal notions, are independent of \ZFC\ (provided they are consistent).}
    \item {Many, many combinatorial and other set-theoretic statements are independent of \ZFC.}
\end{itemize}
\end{example}

\subsection{Building set-theoretic universes}
\label{Subsection.BuildingUniverses}

How are independence results in set theory carried out? We can use our experience with Euclid's postulates as a recipe.

Recipe for building models of alternative geometries:
\begin{itemize}
\item {Begin with a structure we know (e.g. the unit disk).}
\item {Provide alternative definitions of the primitive notions (point, line).}
\item {Verify that the resulting structure satisfies the axioms.}
\end{itemize}

The situation in set theory is analogous.  While there is only one primitive notion (sets), it is harder to begin with an existing structure without biting our own tail, as all of our existing structures such as the unit disk, Euclidean space, and so on are, themselves, already realized as sets.  Furthermore,  given a proposed new structure, the prospect of verifying the axioms is daunting - the axiom system \ZFC\ is incredibly rich, consisting of seven axioms together with two infinite lists of axioms (``axiom schemes'').  This clearly cannot be approached piecemeal, one axiom at a time!

Nonetheless, there are several strategies for constructing new models of set theory.  Of these, one method stands out from the rest for its flexibility: the technique of \emph{forcing}. Invented in 1964 by Paul Cohen, forcing is a method for building a new ``universe of sets,'' or model of \ZFC, out of an existing universe $V$, by selecting a partially-ordered set $\P \in V$ and adjoining a new ideal object called a generic set $G \subset \P$, resulting in the forcing extension $V[G]$ (Fig. \ref{fig.Forcing1} and \ref{fig.Forcing2}).

\begin{figure}[ht]
  \centering
  \begin{minipage}[t]{0.4\textwidth}
        \begin{tikzpicture}[line cap=round,line join=round,>=triangle 45,x=1.0cm,y=1.0cm]
        \clip(0,0) rectangle (4,4);
        \draw (1.,3.5)-- (2.,0.5);
        \draw (2.,0.5)-- (3.,3.5);
        \draw [dash pattern=on 2pt off 2pt] (3.,3.5)-- (1.,3.5);
        \draw (3.1,3.75) node {V};
        \draw (2.2,2.3) node {$\mathbb{P}$};
        \begin{scriptsize}
        \draw [fill=black] (1.9,2.1) circle (2.5pt);
        \end{scriptsize}
        \end{tikzpicture}
    \caption{Select a partial order $\P$.}
    \label{fig.Forcing1}
  \end{minipage}
  \hfill
  \begin{minipage}[t]{0.4\textwidth}
        \begin{tikzpicture}[line cap=round,line join=round,>=triangle 45,x=1.0cm,y=1.0cm]
        \clip(0,0) rectangle (4.7,4);
        \draw (1.,3.5)-- (2.,0.5);
        \draw (2.,0.5)-- (3.,3.5);
        \draw (2.2,2.3) node {$\mathbb{P}$};
        \draw (0.2,3.5)-- (2.,0.5);
        \draw (2.,0.5)-- (3.8,3.5);
        \draw [dash pattern=on 2pt off 2pt] (0.2,3.5)-- (3.8,3.5);
        \draw (3.7,3.7) node[anchor=west] {V[G]};
        \draw (2.85,2.35) node {G};
        \begin{scriptsize}
        \draw [fill=black] (1.9,2.1) circle (2.5pt);
        \draw [fill=black] (2.7,2.1) circle (2.5pt);
        \end{scriptsize}
        \end{tikzpicture}
    \caption{Adjoin the generic set $G$ to build the extension $V[G]$.}
    \label{fig.Forcing2}
  \end{minipage}
\end{figure}

By controlling the properties of the generic set $G$, we can control the properties of the resulting extension $V[G]$, and forcing provides a concrete method for exerting fine control over $G$ via the partial order $\P$. A typical independence result in set theory might proceed as follows:  

Starting in the universe of sets $V$, we build
\begin{itemize}
\item {an extension $V[G_1]$ satisfying the \ZFC\ axioms, in which the continuum hypothesis holds, and}
\item {another extension $V[G_2]$ satisfying the \ZFC\ axioms, in which the continuum hypothesis fails.} 
\end{itemize}
Thus we demonstrate the independence of the continuum hypothesis from the \ZFC\ axioms.

In practice, it is often the case that one of the two parts of such a proof can be accomplished without forcing (as was historically the case with the continuum hypothesis, as \Godel's constructible universe provided a model in which it was true, well before the development of forcing).  But in nearly every independence proof in set theory, forcing is indispensable for at least one of the two parts.

\subsection{A profusion of set theories}
\label{Subsection.Multiverse}

Fifty years of forcing has demonstrated the surprising efficacy of the method in producing diverse models of set theory.  Each such universe of sets exhibits its own distinct mathematical truths.  

\begin{question}
    Is there a compelling reason to accept one of these models of set theory over the others?
\end{question}

This question is a contentious one in modern set theory, with a number of competing views and active research programs supporting them (an excellent informal account of these views is given in Hamkins' response, on MathOverflow, to the question ``What is the Current Status of the Continuum Hypothesis?'' \cite{HamkinsMO:StatusofCH}).  However, no model of set theory yet produced has gained sufficient support within the community to be readily accepted as the ``right model.''  Indeed, not even a single additional axiom has gained serious consideration by the set theory community as a natural addition to \ZFC.  

Let us consider a typical day in the life of a hypothetical modern set theorist - arising in the morning, she might posit a model of set theory $V$ with certain properties, move to a forcing extension $V[G]$, and then to an inner model $W$ of $V[G]$ with some desirable properties of interest (established by carefully tracing the relationships between the models $V$, $V[G]$ and $W$).  This kind of ready movement between different universes of set theory has become familiar, even comfortable, and fails to provoke any philosophical discomfort on the part of the set theorist.  The orthodox view of set theory seems ill-equipped to frame the apparent ease with which we travel between universes - a new paradigm is suggested, exemplified by the multiverse view, as introduced by Hamkins:

\begin{definition}[Multiverse View \cite{HamkinsMO:StatusofCH}, \cite{Hamkins2012:TheSet-TheoreticMultiverse}]
This is the view ... that we do not have just one concept of set leading to a unique set-theoretic universe, but rather a complex variety of set concepts leading to many different set-theoretic worlds.
\end{definition}

\subsection{Exploring the multiverse}
\label{Subsection.MultiverseisSecondOrder}

If we take the multiverse view seriously, how are we to go about understanding this profusion of set theories?  An obvious approach is to reject the study of a universe of sets, and instead direct our attention to studying the set-theoretic multiverse itself.  Taking the multiverse as our object of study presents a challenge, however - as a structure, it is second-order in nature, consisting not simply of primitive objects (``sets'') but also of proper classes of those objects (``universes'').  Our carefully honed tools of first-order set theory are insufficient to the task in its full generality.

Why not simply proceed with the project, second-order logic and all?  Educated as I was in the orthodoxy of first-order set theory, I find this question surprising - I absorbed rather early the notion that ``the second-order'' was difficult, unmanageable and probably dangerous, largely due to the practical limitations (including, but not limited to: the failure of compactness, of completeness, and of the \Lowenheim-Skolem theorem).  I note with some interest (and self-reflective amusement) that, in fact, the debate around second-order logic is alive and ongoing, much more so among philosophers than mathematicians (see, e.g., Simons \cite{Simons1993-SIMWAO-2}).  Nonetheless, we will leave that can of worms half-open for now, in light of the following observation.

It turns out that in certain cases an interesting local neighborhood of the multiverse may be first-order accessible.  The area of set theory known as set-theoretic geology studies just such a local neighborhood.

\section{Set-theoretic geology}
\label{Section.Geology}

Set-theoretic geology is an area of set theory introduced by Gunter Fuchs, Joel David Hamkins, and me in \cite{FuchsHamkinsReitz2015:Set-theoreticGeology}.  Ultimately, it is about forcing, and forcing is a ubiquitious and powerful tool for navigating about the multiverse.  However, the sense in which forcing might give rise to a ``first-order accessible'' neighborhood of the multiverse is not at all obvious - indeed, when we consider all those universes $V[G]$ that might be accessible from our own via forcing, we are already transcending our given domain of discourse (as the set $G$ is not in our universe $V$) and thus transcending the first-order descriptive power at our disposal.  To understand the first-order nature of the geology project requires a different approach.  At its heart, forcing is a method that allows us to move from a model $V$ (the ground model) to a larger model $V[G]$ (the forcing extension). However, a change in perspective allows us to use forcing to look inward.  

\begin{definition}\cite{FuchsHamkinsReitz2015:Set-theoreticGeology}
A class $W\subset V$ is a \emph{ground} of $V$, if $W$ is a transitive class model of \ZFC\ and $V$ is obtained by set forcing over $W$, that is, if there is some forcing notion $\P\in W$ and a $W$-generic filter $G\of\P$
such that $V=W[G]$.
\end{definition}

For a given model $V$, the collection of all its grounds forms the context for \emph{set-theoretic geology}.  This second-order collection, consisting of many proper classes $W$, nonetheless admits a first-order definition. Set-theoretic geology provides first-order access to the multiverse of grounds.  In the following section I will give a brief tour of this neighborhood, point out some of the interesting features, and pose a number of questions.


\subsection{Geology is first-order}
\label{Subsection.GeologyisFirstOrder}

In set-theory, our objects are \emph{sets}.  The idea that a collection of sets is, again, a set, is a natural and attractive notion - but as Russell showed us with his famous paradox, pursuing this notion to its logical conclusion (``unrestricted comprehension'') leads us to contradiction.  Thus we have what amounts to a size restriction on sets - collections of sets that are simply too big are disallowed.  The universe $V$ is one such collection, the collection of all ordinals is another, and so on.  While these collections are not part of our formal universe of discourse - we cannot reference them directly, our quantifiers do not range over them, etc. - we can and do refer to them informally as \emph{proper classes}.  

We also have a more formal method of accessing many proper classes, by using the expressive power of first-order logic. For any first order formula $\phi(x)$ with a free variable, the collection $\{ a\in V \mid \phi(a) \}$ is a definable class.  Since many different interesting properties can be expressed by first-order formulas, we can (on a case-by-case basis) work with many interesting proper classes, all without leaving the realm of first-order logic.

\begin{example}[Some examples of definable classes]
\begin{itemize}
    \item The universe $V$, defined by $\phi(x): x=x$
    \item The class of all of sets that contain the empty set as a member, $\phi(x): \emptyset \in x$
    \item The class of all ordinals (here the definition is a little trickier, and there are several equivalent versions - one is $\phi(x): x$ is a transitive set of transitive sets)
    \item The class of all groups (or fields, or rings, etc)
    \item The class of all models of the parallel postulate.
\end{itemize}
\end{example}
In some cases, the first-order definition comes readily to mind - in others, it does not seem so obvious.  In some cases, we may have a class in mind but, absent a definition, we do not know whether we can address it in a first-order fashion. In the case of set-theoretic geology the classes we are interested in are the grounds of $V$, and the first-order definability of these classes was not established until some fifty years after forcing was invented.

\begin{theorem}[Laver
\cite{Laver2007:CertainVeryLargeCardinalsNotCreated}, indep. Woodin
\cite{Woodin2004:RecentDevelopmentsOnCH}
\cite{Woodin2004:CHMultiverseOmegaConjecture}]
If $W$ is a ground of $V$ then $W$ is a first-order definable class in $V$ (using a parameter from $W$).  
\end{theorem}

The technical details of the proof are beyond the scope of this article, but a few comments about it are in order.  First, the proof actually provides a stronger result than is stated in the theorem.  To understand the importance of this, it's worth considering the problem we are facing - we want to do more than reference a particular ground, we are interested in the study of the collection of \emph{all} the grounds. Knowing only that each ground was first-order definable, we would be able to express statements about \emph{particular} grounds, such as ``the intersection of the grounds $W_1$ and $W_2$ is equal to the ground $W_3$'' (by simply replacing each $W_i$ with the appropriate defining formula).  However, this would not allow us to treat the collection of grounds in a first-order fashion, making statements that (for example) quantified across grounds, such as ``there exists a ground $W$ in which the continuum hypothesis holds'' - exactly the kind of question we want to address in set-theoretic geology.

Luckily, the proof of the theorem does more than establish first-order definability of grounds - a specific first-order formula $\phi(x,r)$ is given in the proof, where $r$ is a set parameter taken from $W$, and so (working in $V$) we have $W = \{ a \in V \mid \phi(a,r) \}$.  The beauty and power of this definition is that it is \emph{uniform} - the same first-order formula will work to define any ground $W$, with the distinction appearing only in our choice of parameter $r$.  Thus, for two different grounds $W_1$ and $W_2$, we can select parameters $r_1 \in W_1$ and $r_2 \in W_2$ such that $W_1$ is defined by $\phi(x,r_1)$ and $W_2$ is defined by $\phi(x,r_2)$.   The uniform nature of the definition allows us to effectively quantify over grounds, by simply quantifying over the parameters $r$ (which, as the parameters are sets, presents no first-order challenges).  Careful examination of the formula $\phi(x,r)$ shows that we can associate to each parameter $r$ a ground $W_r$, and that every ground has a corresponding parameter - this is captured in the following theorem, which establishes the formal foundations for set-theoretic geology.

\begin{theorem}\label{Theorem.ParameterizedGroundsW_r}\cite{FuchsHamkinsReitz2015:Set-theoreticGeology}
There is a parameterized family $\set{W_r\mid r\in V}$ of
classes such that
\begin{enumerate}
 \item Every $W_r$ is a ground of $V$, and $r\in W_r$.
 \item Every ground of $V$ is $W_r$ for some $r$.
 \item The classes $W_r$ are uniformly definable in the
     sense that $\set{\<r,x>\mid x\in W_r}$ is first
     order definable without parameters.
 \item The relation ``$V=W_r[G]$, where $G\of\P\in W_r$
     is $W_r$-generic'' is first-order expressible in
     $V$ in the arguments $(r,G,\P)$.
\end{enumerate}
\end{theorem}

\subsection{A brief tour of geology}
\label{Subsection.TourofGeology}

I think of Theorem \ref{Theorem.ParameterizedGroundsW_r} as a tool that situates set-theoretic geology concretely in the realm of first-order set theory.  What can we do with this tool? What are the questions we can ask, the objects we can construct, the consequences we can explore?  Let's consider some of the possibilities.

If we are interested in the collection of grounds of $V$,  a good place to start is the simplest case possible - what if there are no grounds of $V$?  This is the question that I considered in my dissertation work under Joel David Hamkins \cite{Reitz2006:Dissertation}, focusing on the ground axiom (the precursor to set-theoretic geology).

\begin{definition}[Hamkins, R.]
The \emph{ground axiom} is the assertion that the
universe $V$ is not obtained by set forcing over any
strictly smaller ground model.
\end{definition}

It is worth mentioning here one of the philosophical notions that hovers in the background of the geology project, the search for canonical inner models.  The notion of peering inside a model of set theory, looking for some essential core structure, is a theme that goes back all the way to \Godel's constructible universe and is actively pursued today.  The method of forcing seems to run counter to this inward pull, littering the landscape of our model $V$ with highly non-canonical generic objects $G$.  We can imagine, then, that the ground axiom is a natural property to expect of a canonical inner model - at the least, such a model should not have arisen through forcing over some smaller model (else why not simply take the smaller model?).  Taking it one step further, we can hope that the ground axiom might be a signal that we are in such a minimal, essential model.  Unfortunately, this hope is dashed by the following theorem:

\begin{theorem}[\cite{Reitz2007:TheGroundAxiom,Reitz2006:Dissertation}]
Every model $V$ of set theory has an extension $\Vbar$ which satisfies the ground axiom.
\end{theorem}

A note about the proof: it is clear that $\Vbar$ cannot be constructed from $V$ by set forcing (else it would have a ground model, $V$ itself, and thus the ground axiom would not hold).  Instead, we construct $\Vbar$ using the more powerful (but less tractable) idea of \emph{class forcing}, in which our partial order is a proper class.



If the ground axiom fails, then set-theoretic geology has a nontrivial domain of discourse and we can begin to ask questions about the structure of the grounds.  For example, we can ask what happens if we take ``two steps down:''

\begin{question}
    Is a ground of a ground also a ground?
\end{question}

The answer is yes, and the proof is provided by the long-established theory of \emph{iterated forcing}, a very beautiful notion which shows us how to combine two steps of forcing into a single step.

More generally, what if we take two steps down but ``in different directions,'' - that is, what if we take two grounds and intersect them.  Do we get a ground? Or, easing the restriction a bit, can we find a ground contained in their intersection? We refer to this property as downward directedness.

\begin{question}\cite{Reitz2006:Dissertation}
    Are the grounds downward directed? That is, for any grounds $W_r$ and $W_s$ does there exist a ground $W_t \subset W_r \cap W_s$?    
\end{question}

There is a stronger notion that allows for intersections of more grounds, that of downward set-directedness.

\begin{question}\cite{FuchsHamkinsReitz2015:Set-theoreticGeology}
    Are the grounds downward set-directed? That is, for any set of parameters $A$ does there exist a ground $W_t \subset \bigcap_{r\in A} W_r$?    
\end{question}

These are fundamental questions about the structure of grounds, with wide-ranging implications.  They remained stubbornly open despite our best efforts to resolve them (until very recently - see the next section).

Note that all grounds have a great many elements in common - for example, the ordinals are in every ground, as are all of the constructible sets.  How far does this go? What can we say about the common elements of all the grounds?  These questions give rise to the notion of the mantle, one of the primary objects of interest in set-theoretic geology.

\begin{definition}\cite{FuchsHamkinsReitz2015:Set-theoreticGeology}
The \emph{mantle} $\Mantle$ of a model of set theory is the
intersection of all of its grounds.
\end{definition}

The mantle is a first-order definable class, as $\Mantle = \{x \mid \forall r \, x\in W_r \}$.  It is transitive and contains all ordinals and all constructible sets.  What more can we say about it? It is, again, a natural candidate for a canonical inner model, and so we might expect it to have some nice structural or combinatorial properties.  The following theorem contradicts this expectation entirely, showing that literally \emph{anything} that can occur in a model of \ZFC, can also occur in the mantle.

\begin{theorem}\cite{FuchsHamkinsReitz2015:Set-theoreticGeology}
Every model of \ZFC\ is the mantle of another model - that is, every model $V$ has an extension $\Vbar$ in which the mantle is exactly $V$, so $\Mantle^\Vbar = V$.
\end{theorem}

This theorem shows what is \emph{possible} in the mantle - but can we say anything about what must necessarily hold there?  Is the mantle, for example, necessarily a model of \ZFC?  Interestingly, the answer seems to be tied up in questions of directedness.

\begin{theorem}\cite{FuchsHamkinsReitz2015:Set-theoreticGeology}
\begin{enumerate}
\item If the grounds are downward directed, then the mantle is a model of \ZF.
\item If the grounds are downward set-directed, then the mantle is a model of \ZFC.
\end{enumerate}
\end{theorem}

\subsection{Recent developments in geology}
\label{Subsection.RecentDevelopmentinGeology}

In Fall 2015 the set theorist Toshimichi Usuba announced a result that resolves a great many of the fundamental open questions in set-theoretic geology.  While the result is not yet published, it has generated a great deal of interest in the set theory community and has featured in conference presentations by Usuba and others, notably Hamkins, in a number of international forums.

\begin{theorem}[Usuba]
The grounds are downward set-directed.
\end{theorem}

This implies, among many other things, that the mantle is a model of \ZFC, that the mantle is equal to the generic mantle, that the mantle is equal to the intersection of the generic multiverse of $V$, and that every model in the generic multiverse is no more than two steps away from $V$ itself (an extension of a ground of $V$).

\section{Conclusion}

Set-theoretic geology is of interest because of what it reveals about the relationships between models in a local neighborhood of the set-theoretic multiverse, shedding light onto the structure and properties of the technique of forcing.  It also presents an appealing template for first-order exploration of second-order phenomena - a loose extension of the geology project might be the investigation of other interesting neighborhoods of the multiverse that are amenable to first-order analysis.  A more ambitious extension is the direct study of the multiverse as a whole, building the tools necessary for the task of uncovering the interrelationships between wildly disparate models. Whatever the success of such a program, the multitude of universes revealed by the phenomenon of independence presents a rich and complex landscape, ripe for ongoing exploration.

\begin{acknowledgements}
Thanks to Mihir Chakraborty and Michele Friend for their generous and inclusive approach to mathematics and philosophy, truly the embodiment of pluralism, and to Joel David Hamkins for his ongoing support and shared excitement in the work.
\end{acknowledgements}

\bibliographystyle{spmpsci}      
\bibliography{HamkinsBiblio,LogicBiblio,OtherBiblio}   

\begin{thebibliography}{10}
\providecommand{\url}[1]{{#1}}
\providecommand{\urlprefix}{URL }
\expandafter\ifx\csname urlstyle\endcsname\relax
  \providecommand{\doi}[1]{DOI~\discretionary{}{}{}#1}\else
  \providecommand{\doi}{DOI~\discretionary{}{}{}\begingroup
  \urlstyle{rm}\Url}\fi

\bibitem{Beltrami1868}
Beltrami, E.: Teoria fondamentale degli spazii di curvatura costante.
\newblock Annali di Matematica Pura ed Applicata (1867-1897) \textbf{2}(1),
  232--255 (1868).
\newblock \doi{10.1007/BF02419615}.
\newblock \urlprefix\url{http://dx.doi.org/10.1007/BF02419615}

\bibitem{DavisHersh1981:TheMathematicalExperience}
Davis, P.J., Hersh, R.: The Mathematical Experience.
\newblock Houghton Mifflin, Boston and New York (1981)

\bibitem{FuchsHamkinsReitz2015:Set-theoreticGeology}
Fuchs, G., Hamkins, J.D., Reitz, J.: Set-theoretic geology.
\newblock Annals of Pure and Applied Logic \textbf{166}(4), 464--501 (2015).
\newblock \doi{10.1016/j.apal.2014.11.004}.
\newblock \urlprefix\url{http://jdh.hamkins.org/set-theoreticgeology}

\bibitem{Hamkins2012:TheSet-TheoreticMultiverse}
Hamkins, J.D.: The set-theoretic multiverse.
\newblock Review of Symbolic Logic \textbf{5}, 416--449 (2012).
\newblock \doi{10.1017/S1755020311000359}.
\newblock \urlprefix\url{http://jdh.hamkins.org/themultiverse}

\bibitem{HamkinsMO:StatusofCH}
(http://mathoverflow.net/users/1946/joel-david hamkins), J.D.H.: Solutions to
  the continuum hypothesis.
\newblock MathOverflow.
\newblock \urlprefix\url{http://mathoverflow.net/q/25199}.
\newblock URL:http://mathoverflow.net/q/25199 (version: 2010-05-19)

\bibitem{Hilbert:FoundationsOfGeometry}
Hilbert, D.: Foundations of Geometry.
\newblock Open Court (1999).
\newblock
  \urlprefix\url{http://www.amazon.com/Foundations-Geometry-David-Hilbert/dp/0875481647%3FSubscriptionId%3D0JYN1NVW651KCA56C102%26tag%3Dtechkie-20%26linkCode%3Dxm2%26camp%3D2025%26creative%3D165953%26creativeASIN%3D0875481647}

\bibitem{Laver2007:CertainVeryLargeCardinalsNotCreated}
Laver, R.: Certain very large cardinals are not created in small forcing
  extensions.
\newblock Annals of Pure and Applied Logic \textbf{149}(1), 1 -- 6 (2007).
\newblock \doi{http://dx.doi.org/10.1016/j.apal.2007.07.002}.
\newblock
  \urlprefix\url{http://www.sciencedirect.com/science/article/pii/S0168007207000607}

\bibitem{Reitz2006:Dissertation}
Reitz, J.: The ground axiom.
\newblock Ph.D. thesis, The Graduate Center of the City University of New York,
  365 Fifth Avenue, New York, NY 10016 (2006)

\bibitem{Reitz2007:TheGroundAxiom}
Reitz, J.: The ground axiom.
\newblock The Journal of Symbolic Logic \textbf{72}(4), 1299--1317 (2007)

\bibitem{Simons1993-SIMWAO-2}
Simons, P.: Who's afraid of higher-order logic?
\newblock Grazer Philosophische Studien \textbf{44}, 253--264 (1993)

\bibitem{Woodin2004:RecentDevelopmentsOnCH}
Woodin, W.H.: Recent developments on {Cantor's Continuum Hypothesis}.
\newblock Proceedings of the Continuum in Philosophy and Mathematics  (1004).
\newblock Carlsberg Academy, Copenhagen, November 2004

\bibitem{Woodin2004:CHMultiverseOmegaConjecture}
Woodin, W.H.: The continuum hypothesis, the generic-multiverse of sets, and the
  {$\Omega$} conjecture.
\newblock In: J.~Kennedy, R.~Kossak (eds.) Set Theory, Arithmetic, and
  Foundations of Mathematics, pp. 13--42. Cambridge University Press (2011).
\newblock \urlprefix\url{http://dx.doi.org/10.1017/CBO9780511910616.003}.
\newblock Cambridge Books Online

\end{thebibliography}

%
%

\end{document}